\documentclass[12pt]{article}

\usepackage{fullpage,float,amssymb,amsmath,amsthm}
\usepackage{epsf}
\usepackage[usenames]{color}
\usepackage[colorlinks=true,
linkcolor=webgreen,
filecolor=webbrown,
citecolor=webgreen]{hyperref}

\definecolor{webgreen}{rgb}{0,.5,0}
\definecolor{webbrown}{rgb}{.6,0,0}
\def\divides{\, | \,}

\def\modd#1 #2{#1\ \mbox{\rm (mod}\ #2\mbox{\rm )}}
\DeclareMathOperator{\lcm}{lcm}
\DeclareMathOperator{\ord}{ord}

\usepackage{color}
\usepackage{enumerate}

\author{Daniel M. Kane\thanks{Supported by NSF Award CCF-1553288
(CAREER) and a Sloan Research Fellowship.}\\
Mathematics \\
University of California, San Diego \\
9500 Gilman Drive \#0404\\
La Jolla, CA  92093-0404 \\
USA \\
{\tt dakane@math.ucsd.edu} \\
\and
Carlo Sanna \\
Dipartimento di Matematica ``Giuseppe Peano''\\
Universit{\`a} degli Studi di Torino\\
Via Carlo Alberto 10 \\
10123 Torino \\ 
Italy \\
{\tt carlo.sanna.dev@gmail.com}\\
\and
Jeffrey Shallit\thanks{Supported by NSERC Discovery Grant \#105829/2013.}\\
School of Computer Science\\
University of Waterloo\\
Waterloo, ON  N2L 3G1 \\
Canada\\
{\tt shallit@uwaterloo.ca} 
}

\title{Waring's Theorem for Binary Powers}

\begin{document}

\maketitle

\theoremstyle{plain}
\newtheorem{theorem}{Theorem}
\newtheorem{corollary}[theorem]{Corollary}
\newtheorem{lemma}[theorem]{Lemma}
\newtheorem{proposition}[theorem]{Proposition}

\theoremstyle{definition}
\newtheorem{definition}[theorem]{Definition}
\newtheorem{example}[theorem]{Example}
\newtheorem{conjecture}[theorem]{Conjecture}

\theoremstyle{remark}
\newtheorem{remark}[theorem]{Remark}

\def\Que{{\mathbb{Q}}}
\def\Enn{{\mathbb{N}}}
\def\Zee{{\mathbb{Z}}}
\newcommand{\seqnum}[1]{\href{http://oeis.org/#1}{\underline{#1}}}
\def\mS{{\mathcal{S}}}
\def\mT{{\mathcal{T}}}

\begin{abstract}
A natural number is a \emph{binary $k$'th power} if its binary representation consists of $k$ consecutive identical blocks.
We prove an analogue of Waring's theorem for sums of binary $k$'th powers.
More precisely, we show that for each integer $k \geq 2$, there exists a positive integer $W(k)$ such that every sufficiently large multiple of $E_k := \gcd(2^k - 1, k)$ is the sum of at most $W(k)$ binary $k$'th powers.
(The hypothesis of being a multiple of $E_k$ cannot be omitted, since we show that the $\gcd$ of the binary $k$'th powers is $E_k$.)
Also, we explain how our results can be extended to arbitrary integer bases $b > 2$.
\end{abstract}

\section{Introduction}

Let $\Enn = \{ 0,1,2,\ldots \}$ be the natural numbers and let $S \subseteq \Enn$.   The principal problem of additive number theory is to determine whether every integer $N$ (resp., every sufficiently large integer $N$) can be represented as the sum of some {\it constant\/} number of elements of $S$, not necessarily distinct,
where the constant does not depend on $N$.  For a superb introduction to this topic, see \cite{N1}.

Probably the most famous theorem of additive number theory is Lagrange's theorem from 1770: every natural number is the sum of four squares \cite{L}.   Waring's problem
(see, e.g., \cite{Sm,VW}), first stated by Edward Waring in 1770, is to determine $g(k)$ such that every natural number is the sum of $g(k)$ $k$'th powers.  (A priori, it is not even clear that $g(k) < \infty$, but this was proven by Hilbert in 1909.)  From Lagrange's theorem we know that $g(2) = 4$.   For other results concerning sums of squares, see, e.g., \cite{G,MW}.

If every natural number is the sum of $k$ elements of $S$,
we say that $S$ forms a {\it basis} of order $k$.  If every sufficiently large natural number is the sum of $k$ elements of $S$, we say that $S$ forms an {\it asymptotic basis} of order $k$.

In this paper, we consider a variation on Waring's theorem, where the ordinary notion of integer power is replaced by a related notion inspired from formal language theory.  Our main
result is Theorem~\ref{main} below.  We say that a natural number $N$ is 
a {\it base-$b$ $k$'th power} if its base-$b$ representation consists of $k$ consecutive identical blocks.  For example, 3549 in base $2$ is
	$$ 1101 \, 1101 \, 1101 ,$$
so 3549 is a base-2 (or binary) cube.  Throughout this paper, we  consider only the {\it canonical} base-$b$
expansions (that is, those without leading zeros).  The binary squares
$$ 0,3,10,15,36,45,54,63,136,153,170,187,204,221,238,255,528,561,594,627,\ldots$$
form sequence \seqnum{A020330} in Sloane's {\it On-Line Encyclopedia of Integer Sequences} \cite{Sl}.  The binary cubes 
$$ 0,7,42,63,292,365,438,511,2184,2457,2730,3003,3276,3549,3822,4095,16912,\ldots$$
form sequence \seqnum{A297405}. 

Notice that a number $N>0$ is a base-$b$ $k$'th power if and only if
we can write $N = a \cdot c_k^b (n)$, where 
$$c_k^b (n) := \frac{b^{kn}-1}{b^n - 1} = 1 + b^n + \cdots + b^{(k-1)n}$$
for some $n \geq 1$ such that
$b^{n-1} \leq a < b^n$.  (The latter condition is needed to ensure that the base-$b$ $k$'th power is formed by the concatenation of blocks that begin with a nonzero digit.) Such a number consists of
$k$ consecutive blocks of digits, each of length $n$.
For example, $3549 = 13 \cdot c_3^2 (4)$. 
We define
$$\mS_k^b := \left\{ n \geq 0  \ : \ n \text{ is a base-$b$ $k$'th power} \right\} = \left\{ a \cdot c_k^b (n)  \ : \ n \geq 1, \ b^{n-1} \leq a < b^n \right\}.$$

The set $\mS_k^b$ is an interesting and natural set to study because its counting function is $\Omega(N^{1/k})$, just like the ordinary $k$'th powers.  It has also appeared in a number of recent papers
(e.g., \cite{BLS}).  However, there are two significant differences between the ordinary $k$'th powers and the base-$b$ $k$'th powers.

The first difference is that $1$ is not a base-$b$ $k$'th power for $k > 1$.  Thus, the base-$b$ $k$'th powers cannot, in general, form a basis of finite order, but only an asymptotic basis.

A more significant difference is that 
the gcd of the ordinary $k$'th powers is always equal to $1$, while the gcd of the base-$b$ $k$'th powers may, in some cases, be greater than one.  This is
quantified in Section~\ref{gcd2}.  Thus, it is not reasonable to expect that every sufficiently large natural number can be the sum of a fixed number of
base-$b$ $k$'th powers; only those that are also a multiple of the $\gcd$ can be so represented.   

\section{The greatest common divisor of $\mS_k^b$}
\label{gcd2}

\begin{theorem}
For $k \geq 1$ define
\begin{align*}
A_k &= \gcd( \mS_k^b) ,\\
B_k &= \gcd( c_k^b (1), c_k^b (2), \ldots) , \\
C_k &= \gcd( c_k^b (1), c_k^b (2), \ldots, c_k^b (k)) , \\
D_k &= \gcd( c_k^b (1), c_k^b (k)) , \\
E_k &= \gcd\!\left( \frac{b^k - 1}{b-1},  k\right) .
\end{align*}
Then $A_k = B_k = C_k = D_k = E_k$.
\label{sanna}
\end{theorem}

\begin{proof}

\noindent $A_k = B_k$:   If $d$ divides
$B_k$, then it clearly also divides
all numbers of the form $a \cdot c_k^b (n)$ with
$b^{n-1} \leq a < b^n$ and hence $A_k$.  

On the other hand if
$d$ divides $A_k$, then it divides $c_k^b (1)$.
Furthermore, $d$ divides
$b^{n-1} \cdot c_k^b (n)$ and
$(b^{n-1} + 1) c_k^b (n)$ (both of which are members
of $\mS_k^b$ provided $n \geq 2$).  So it must divide
their difference, which is just $c_k^b (n)$.  So $d$ divides $B_k$.  

\medskip

\noindent $B_k = C_k$:
Note that $d$ divides $B_k$
if and only if it divides $c_k^b (1)$ and also
$c_k^b (n) \bmod c_k^b (1)$ for all $n \geq 1$.
Now it is well known that, for $b \geq 2$ and
integers $n, k \geq 1$, we have
$$b^n \equiv \modd{b^{n \bmod k}} {b^k - 1}.$$
Hence
\begin{align*}
c_k^b (n) &=
1 + b^n + \cdots + b^{(k-1)n} \equiv
\modd{1 + b^{n \bmod k} + \cdots + b^{(k-1)n \bmod k}} {b^k - 1} \\
&\equiv \modd{1 + b^a + \cdots + b^{(k-1)a}} {b^k-1}  \\
& \equiv \modd{1 + b^a + \cdots + b^{(k-1)a}} {c_k^b(1)} \\
& \equiv \modd{c_k^b (a)} {c_k^b (1)},
\end{align*}
where $a = n \bmod k$.  Thus any divisor of $C_k$
is also a divisor of $B_k$.  The converse is clear.
\medskip

\noindent $D_k = E_k$:  It suffices to observe that
\begin{align*}
c_k^b (k) &= 1 + b^k + \cdots + b^{(k-1)k} \\
&\equiv \modd{\overbrace{1 + 1 + \cdots + 1}^k} {b^k - 1} \\
&\equiv \modd{k} {b^k - 1} \\
& \equiv \modd{k} {\frac{b^k-1}{b-1}} \\
& \equiv \modd{k} {c_k^b (1)}.
\end{align*}

\medskip

\noindent $B_k = E_k$:  Every divisor of $B_k$
clearly divides $D_k$, and above we saw $D_k = E_k$.  We now show that every prime divisor
of $E_k$ divides $B_k$ to at least the same order, thus showing
that every divisor of $E_k$ divides $B_k$.  
We need the following classic lemma, sometimes called the ``lifting-the-exponent'' or LTE lemma \cite{C}:

\begin{lemma}\label{lem:LTE}
If $p$ is a prime number and $c \neq 1$ is an integer such that $p \mid c - 1$, then
\begin{equation*}
\nu_p\!\left(\frac{c^n - 1}{c - 1}\right) 
\geq \nu_p(n) ,
\end{equation*}
for all positive integers $n$, where $\nu_p (n)$ is the $p$-adic valuation of $n$ (the exponent of the highest power of $p$ dividing $n$).  
\end{lemma}

Fix an integer $\ell \geq 1$ and let $p$ be a prime factor of $E_k$.
On the one hand, if $p \mid b^\ell - 1$, then by Lemma~\ref{lem:LTE} we get that
\begin{equation*}
\nu_p\!\left(c_k^b(\ell)\right)=\nu_p\!\left(\frac{b^{k\ell} - 1}{b^\ell - 1}\right) \geq \nu_p(k) \geq \nu_p(E_k) ,
\end{equation*}
since $E_k \mid k$.  Hence $p^{\nu_p(E_k)} \mid c_k^b(\ell)$.
On the other hand, if $p \nmid b^\ell - 1$, then $p^{\nu_p(E_k)}$ divides $c_k^b(\ell) = \frac{b^{k\ell} - 1}{b^\ell - 1}$ simply because $p^{\nu_p(E_k)}$ divides the numerator but does not divide the denominator.
In both cases, we have that $p^{\nu_p(E_k)} \mid c_k^b(\ell)$, and since this is true for all prime divisors of $E_k$, we get that $E_k \mid c_k^b(\ell)$, as desired.
\end{proof}

\begin{remark}
For $b = 2$, the sequence $E_k$ is sequence \seqnum{A014491} in
Sloane's {\it Encyclopedia}.  We make some additional remarks
about the values of $E_k$ in Section~\ref{final}.
\end{remark}

In the remainder of the paper, for concreteness,
we focus on the case $b = 2$.  We set
$c_k (n) := c_k^2 (n)$ and $\mS_k := \mS_k^2$.    However, everything we say also
applies more generally to bases $b > 2$, with one minor
complication that is mentioned in Section~\ref{final}.

\section{Waring's theorem for binary $k$'th powers: proof outline and tools}

We now state the main result of this paper.

\begin{theorem}
Let $k \geq 1$ be an integer.  Then there is a number
$W(k) < \infty$ such that every sufficiently large multiple of $E_k = \gcd(2^k - 1,k)$
is representable as the sum of at most $W(k)$ binary
$k$'th powers.
\label{main}
\end{theorem}

\begin{remark}
The fact that $W(2) \leq 4$ was proved in
\cite{MNRS}.
\end{remark}

\begin{proof}[Proof sketch]
Here is an outline of the proof.  All of the mentioned constants depend only on $k$. 

Given a number $N$, a multiple of $E_k$, 
that we wish to represent as a sum of binary $k$'th powers, 
we first
choose a suitable power of $2$, say $x = 2^n$, and think of $N$ as a degree-$k$ polynomial $p$ evaluated
at $x$. For example, we
can represent $N$ in base $2^n$; the ``digits'' of this
representation
then correspond to the coefficients of $p$.

Similarly, the integers $c_k (n), c_k(n+1), \ldots,
c_k (n+k-1)$ can also be viewed as polynomials
in $x = 2^n$.  By linear algebra, there is a unique way
to rewrite $p$ as a linear combination of
$c_k (n), c_k(n+1), \ldots,
c_k (n+k-1)$, and this linear transformation can be
represented by a matrix $M$ that depends only on $k$, and
is independent of $n$.

At first glance, such a linear combination would seem to provide a suitable representation of $N$ in terms of binary $k$'th powers, but there are three problems to overcome: 
\begin{enumerate}[(a)]
\item the coefficients of $c_k (i)$,
$n \leq i < n+k$, could be much too large;
\item the coefficients could be too small or negative;
\item the coefficients might not be integers.
\end{enumerate}

Issue (a) can be handled by choosing $n$ such that $2^n \approx N^{1/k}$.  This guarantees that the resulting coefficients of the $c_k (n)$ are at most a constant factor larger than $2^n$.  Using
Lemma~\ref{split} below, the coefficients can
be ``split'' into at most a constant number of coefficients
lying in the desired range.

Issue (b) is handled by not working with $N$, but rather with
$Y := N - D$, where $D$ is a suitably chosen linear combination of $c_k (n), c_k (n+1), \ldots, c_k(n+k-1)$ with large positive integer coefficients.  Any negative coefficients
arising in the expression for $Y$ can now be offset by adding the large positive coefficients corresponding to $D$, giving us coefficients for the representation of $N$ that are positive and lie in a suitable range.

Issue (c) is handled by rounding down the coefficients of the linear combination to the next lower integer. 
This gives us a representation, as a sum of binary
$k$'th powers, for
some smaller number $N' < N$, where the difference
$N - N'$ is a sum 
of at most $k^2$ terms of the form $2^i/d$, where $d$ is the determinant of $M$.  However, the base-$2$ representation of $1/d$ is, disregarding leading zeros, actually periodic with some period $p$.  By choosing an
appropriate small multiple of a binary $k$'th power corresponding to $k$ copies of this period, we can
approximate each $2^i/d$, and hence $N - N'$,
from below by some number $N''$ that is a sum of binary
$k$'th powers.

The remaining error term is $Q := N - N' - N''$, which
turns out to be at most some constant depending on $k$.
Since $N$ is a multiple of $E_k$ and $N'$ and
$N''$ are sums of binary $k$'th powers, it follows
that $Q$ is also a multiple of $E_k$.  With care we can ensure that
$Q$ is larger than the Frobenius number of the binary $k$'th powers, and
hence $Q$ can be written as a sum of elements of
$\mS_k$.   On the other hand,
since $Q$ is a 
constant, at most a constant number of additional
binary $k$'th powers are needed to represent it.
This completes the sketch of our construction.  It
is carried out in more detail in the rest of the paper.
\end{proof}

\begin{remark}
In what follows, we spend a small amount of time explaining
that certain quantities are actually constants that depend
only on $k$.  By estimating these constants we could
come up with an explicit bound on $W(k)$, but we have not
done so.
\end{remark}

\subsection{Expressing multiples of $c_k (n)$ as
a sum of binary $k$'th powers}

As we have seen, a number of the form $a \cdot c_k (n)$
with $2^{n-1} \leq a < 2^n$ is a binary $k$'th power.
But how about larger multiples of $c_k (n)$?  The
following lemma will be useful.

\begin{lemma}
Let $a \geq 2^{n-1}$.  Then $a \cdot c_k (n)$ is the
sum of at most $\lceil \frac{a}{2^n -1} \rceil$ 
binary $k$'th powers.
\label{split}
\end{lemma}

\begin{proof}
Clearly the claim is true for $2^{n-1} \leq a < 2^n$.  Otherwise,
define $b := \lceil \frac{a}{2^n -1} \rceil$ and
$c := (2^n-1)b - a$, so that $0 \leq c < 2^n - 1$.
Then $a = (b-2)(2^n-1) + d_1 + d_2$,
where $d_1 = \lfloor (2^n - 1) - \frac{c}{2} \rfloor$
and $d_2 = \lceil (2^n - 1) - \frac{c}{2} \rceil$.
A routine calculation now shows 
that $2^{n-1} \leq d_1 \leq
d_2 < 2^n$, and so $a \cdot c_k (n)$ is the sum
of $b$ binary $k$'th powers.
\end{proof}

\subsection{Change of basis and the Vandermonde matrix}
\label{vander}

In what follows, matrices and vectors are always indexed
starting at $0$.
Recall that a Vandermonde matrix 
$$ V(a_0, a_1, \ldots, a_{k-1})$$
is a $k\times k$ matrix where the entry in the $i$'th
row and $j$'th column, for $0 \leq i, j < k$, is
defined to be $a_i^j$.    The matrix is invertible if and only if the
$a_i$ are distinct.

Recall that $c_k (n) = 1 + 2^n + 2^{2n} + \cdots + 2^{(k-1)n}$.  For $k \geq 1$ and $n \geq 0$ we have
\begin{equation}
\left[ \begin{array}{c}
c_k (n) \\
c_k (n+1) \\
\vdots \\
c_k(n+k-1)
\end{array} \right] = M_k
\left[ \begin{array}{c}
1  \\
2^n \\
\vdots \\
2^{(k-1)n} 
\end{array} \right],
\label{kane}
\end{equation}
where $M_k = V(1, 2, 4, \ldots, 2^{k-1})$.   For example,
$$ M_4 = \left[ \begin{array}{cccc} 
1 & 1 & 1 & 1 \\
1 & 2 & 4 & 8 \\
1 & 4 & 16 & 64 \\
1 & 8 & 64 & 512
\end{array}
\right].$$
Let a natural number
$Y$ be represented as an $\Enn$-linear combination 
$$Y = a_0 + a_1 2^n + \cdots + a_{k-1} 2^{(k-1)n}.$$
Then, multiplying Eq.~\eqref{kane} on the left by
$$ [b_0 \quad b_1 \quad \cdots \quad b_{k-1}] := [a_0 \quad a_1 \quad \cdots \quad a_{k-1}] M_k^{-1},$$
we get the following expression for $Y$ as a $\Que$-linear
combination of binary $k$'th powers:
$$ Y = b_0 c_k (n) + b_1 c_{k}(n+1) + \cdots +
	b_{k-1} c_k(n+k-1) .$$
It remains to estimate the size of the coefficients $b_i$,
as well as the sizes of their denominators.

The Vandermonde matrix is well studied (e.g.,
\cite[pp.~43, 105]{PS}).  We recall one
basic fact about it.

\begin{lemma}
The determinant of $V(a_0, a_1, \ldots, a_{k-1})$ is
$$ \prod_{0 \leq i < j < k} (a_j - a_i) .$$
\label{van}
\end{lemma}

We now define $d_k$ to be the determinant of $M_k$,
and $\ell_k$ to be the largest of the absolute values of the entries of $M_k^{-1}$. Note that, by Lemma~\ref{van}, $d_k$ is positive.
Also, Laplace's formula tells us that $M_k^{-1} = M'_k d_k^{-1}$, where
$M'_k$ is the adjugate (classical adjoint) $M'_k$ of $M_k$. Furthermore, since $M_k$ has integer entries, so does $M'_k$.

\begin{proposition}
We have $0 < d_k < 2^{k^3/3}$ for $k \geq 1$.
\end{proposition}

\begin{proof}
By the formula of Lemma~\ref{van} we know
that 
$$d_k = \prod_{0 \leq i < j< k} (2^j - 2^i)
< \prod_{0 \leq i < j< k} 2^j = 2^{k^3/3 -k^2/2 + k/6} 
< 2^{k^3/3}$$
for $k \geq 1$.
\end{proof}

Our next result demonstrates that $\ell_k$, the
absolute value of the largest entry in
$M_k^{-1}$, is bounded above by a constant.

\begin{proposition}
We have $\ell_k < 34$.
\end{proposition}

\begin{proof}
As is well known 
(see, e.g., \cite[Exercise 1.2.3.40]{K}, the $i$'th column in the
inverse of the Vandermonde matrix
$V(a_0, a_1, \ldots, a_{k-1})$ consists
of the coefficients of the polynomial 
$$ p(x) := \frac{\prod_{{0 \leq j< k} \atop {j\not=i}}(x-a_i)}{\prod_{{0 \leq j< k} \atop {j\not=i}} (a_j - a_i)}.$$

We also observe that if
$$(x-b_1)(x-b_2) \cdots (x-b_n) =
	x^n + c_{n-1} x^{n-1} + \cdots + c_1x + c_0,$$
is a polynomial with real roots, then the absolute value of every
coefficient $c_i$ is bounded by
$$ |c_0| + \cdots + |c_{n-1}| \leq 	
\prod_{1 \leq i \leq n}  (1 + |b_i|) .$$

Putting these two facts together, we see that all of the entries in the $i$'th column of
$V(a_0, a_1, \ldots, a_{n-1})^{-1}$ are, in absolute
value, bounded by 
$$ P_k (i) := \frac{\prod_{{0 \leq j< k} \atop {j\not=i}}(1+ |a_j|)}{\prod_{{0 \leq j< k} \atop {j\not=i}} |a_j - a_i|}.$$

Now let's specialize to $a_i = 2^i$.  We get
$$
P_k (i)  := \frac{\prod_{{0 \leq j< k} \atop {j\not=i}}(2^j +1)}{\prod_{{0 \leq j< k} \atop {j\not=i}} |2^j - 2^i|} \\
\leq 
\frac{\prod_{0 \leq j< k}(2^j + 1)}{\prod_{{0 \leq j< k} \atop {j\not=i}} |2^j - 2^i |}. 
$$
To finish the proof of the upper bound, it remains to find a lower bound for the denominator
$$ Q_k (i) := \prod_{{0 \leq j < k} \atop {j \not= i}}
|2^j - 2^i|  .$$
We claim, for $k \geq 2$, that
\begin{equation}
Q_k (0) \geq Q_k (1)  \label{ineq1}
\end{equation}
and
\begin{equation}
Q_k (1) \leq Q_k (2) \leq \cdots \leq Q_k (k-1).
\label{ineq2}
\end{equation}

To see \eqref{ineq1}, note that
$Q_k (0) = \prod_{2 \leq j < k} (2^j - 1)$ and
$Q_k (1) = \prod_{2 \leq j < k}  (2^j - 2)$.
On the other hand, by telescoping cancellation we
see, for $1 \leq i \leq k - 2$, that
$$ \frac{Q_k (i)}{Q_k (i+1)} =
\frac{2^{k-1} - 2^i}{(2^{i+1} - 1)2^{k-2}} <
\frac{2^{k-1}}{3 \cdot 2^{k-2}} = \frac23 ,$$
which proves \eqref{ineq2}.   Hence $Q_k (i)$ is minimized at
$i = 1$.
Now
\begin{align*}
\ell_k &\leq \frac{\prod_{0 \leq j < k} (2^j + 1)}{Q_k (i)} 
 \leq \frac{\prod_{0 \leq j < k} (2^j + 1)}{Q_k (1)}  \\[1pt]
& =  \frac{\prod_{0 \leq j < k} (2^j + 1)}{\prod_{2 \leq j < k} (2^j - 2)} 
 < 2 \cdot 3 \cdot \prod_{j \geq 2} \frac{2^j + 1}{2^j - 2}
\doteq 33.023951743\cdots < 34.
\end{align*}
\end{proof}

\begin{remark}
The tightest upper bound seems to be
$\ell_k < 5.194119929183\cdots$ for all $k$, but we
did not prove this.
\end{remark}

\subsection{Expressing fractions of powers of $2$ as
sums of binary $k$'th powers}
\label{fractions}

In everything that follows, $k$ is an integer greater than $1$.

\begin{lemma}
Let $f > 1$ be an odd integer.  Define
$e = \lfloor \log_2 f \rfloor$, so that 
$2^e < f < 2^{e+1}$.  Let $m$ be the order
of $2$ in the multiplicative group of integers
modulo $f$.  Then for all integers $j \geq 1$, the
number
$$ \left\lfloor \frac{2^{jm+e}}{f} \right\rfloor$$
is a binary $j$'th power, whose base-$2$ representation consists of $j$ repetitions of a block of size $m$.
\label{jef5}
\end{lemma}

\begin{proof}
Since $m$ is the multiplicative order of $2$ modulo $f$, we have $2^m - 1 = fq$ for some positive integer $q$.  Then
$$ \frac{1}{f} = \frac{q}{2^m - 1} = q \sum_{i \geq 1} 2^{-im}.$$
Multiplying by $2^{jm+e}$ and splitting the summation into two pieces, we see that
\begin{align}\label{twojmeoverf}
\frac{2^{jm+e}}{f} &= q \cdot 2^{jm+e} \sum_{i \geq 1} 2^{-im} \nonumber\\
&= q \cdot 2^{jm+e} \sum_{1 \leq i \leq j} 2^{-im} +
q \cdot 2^{jm+e} \sum_{i > j} 2^{-im} \nonumber\\
&= q \cdot 2^e \cdot \frac{2^{jm}-1}{2^m - 1} + 
\frac{2^e}{f} . 
\end{align}
Since $2^e < f$, the right-hand side of Eq.~(\ref{twojmeoverf}) is the
sum of an integer and a number strictly between
$0$ and $1$.  It follows that
$$ \left\lfloor \frac{2^{jm+e}}{f} \right\rfloor
	= q \cdot 2^e \cdot \frac{2^{jm}-1}{2^m - 1}.$$
It remains to see that $q \cdot 2^e$ is in the right
range:  we must have $2^{m-1} \leq q \cdot 2^e < 2^m$.

To see this, note that 
$$
q \cdot 2^e = \frac{2^m -1}{f} \cdot 2^e \\
= \frac{2^{m+e}}{f} - \frac{2^e}{f} ,
$$
and, since $0 < 2^e / f < 1$, it follows that
$$ \frac{2^{m+e}}{f} - 1 < q \cdot 2^e < \frac{2^{m+e}}{f} .$$
Rewriting gives
$$ 2^{m-1} - 1 < 2^{m-1} \left( \frac{2^{e+1}}{f} \right) - 1 <  q \cdot 2^e \leq \left( \frac{2^e}{f} \right) 2^m < 2^m,$$
or $2^{m-1} \leq q \cdot 2^e < 2^m$, as desired.
\end{proof}
 
\begin{lemma}
Let $g$ be an integer with $g = 2^\ell \cdot f$,
where $f \geq 1$ is odd.  
Then for all $n \geq kf + \ell + \log_2 f$, the number
$ \lfloor \frac{2^n}{g} \rfloor$ can be written
as the sum of at most $2^{kf-1}$ binary $k$'th powers
and an integer $t$ with $0 \leq t \leq 2^{kf-1}$.
\label{pows}
\end{lemma}

\begin{proof}
There are two cases:  (a) $f = 1$ or (b) $f > 1$.

\medskip

\noindent (a) $f = 1$:  Using the division algorithm write
$n - \ell = rk + i$ for $0 \leq i\leq k-1$.  
Since $n \geq \ell$ we have
$$ \left\lfloor \frac{2^n}{g} \right\rfloor = 
\frac{2^n}{g} = 2^{n-\ell} = 2^{rk+i} = 2^i (2^{rk}- 1) + 2^i.$$
The base-$2$ representation of $2^{rk}-1$ is clearly
a binary $k$'th power.   Take $t = 2^i$. 

\medskip
 
\noindent (b) $f > 1$:  
Let $e = \lfloor \log_2 f \rfloor$ and
let $m$ be the order
of $2$ in the multiplicative group of integers
modulo $f$.

Using the division algorithm, write
$n - \ell - e = rkm + i$ for some $i$ with
$0 \leq i \leq km - 1$.  Note that since
$n \geq kf + \ell + \log_2 f \geq km + \ell + e$
we have $r \geq 1$.  

Then
\begin{align*}
\frac{2^n}{g} &= \frac{2^{rkm+i+\ell+e}}{2^\ell \cdot f}
= \frac{2^{rkm+i+e}}{f} \\
&= 2^i \cdot \frac{2^{rkm+e}}{f} = 
2^i \left\lfloor \frac{2^{rkm+e}}{f} \right\rfloor + t,
\end{align*}
with $0 \leq t < 2^i$.
Now take the floor of both sides and apply Lemma~\ref{jef5}.
\end{proof}

\subsection{The Frobenius number}

Let $S$ be a set and $x$ be a real number.
By $xS$ we mean the set $\{ xs \ : \ s \in S \}$.

Let $S \subseteq \Enn$ with $\gcd(S) = 1$.  The
Frobenius number of $S$, written $F(S)$, is the largest integer that cannot be represented as a non-negative integer linear combination of elements of $S$.   See, for example, \cite{RA}.

As we have seen, $\gcd(\mS_k) = E_k = \gcd(k, 2^k - 1)$.
Thus $\gcd(E_k^{-1} \mS_k) = 1$.  Define $F_k$ to be the
Frobenius number of the set $E_k^{-1} \mS_k$.  In this section we give a weak upper bound for $F_k$.

\begin{lemma}
For $k \geq 2$ we have
$F_k \leq 2^{k^2+k}$.
\label{frobl}
\end{lemma}

\begin{proof}
Consider $T = \{g_1, g_2, g_3 \}$ where
$g_1 = 2^k - 1$, $g_2 = (2^k - 2) \frac{2^{k^2} - 1}{2^k -1}$, and $g_3 = (2^k - 1) \frac{2^{k^2} - 1}{2^k -1}$.  We have
$T \subseteq \mS_k$.    Let $d$ be the greatest
common divisor of
$T$.  Then $d$ divides $g_3 - g_2 = \frac{2^{k^2} - 1}{2^k -1}$ and $g_1 = 2^k - 1$.  So $d$ divides $D_k$. On the other hand, clearly, $A_k$ divides $d$, while from Theorem~\ref{sanna} we know that $A_k = D_k = E_k$. Hence, $d = E_k$.

Clearly $F(E_k^{-1} \mS_k) \leq F(E_k^{-1} T)$.  
Furthermore, since
$g_1 \divides g_3$, it follows that
$F(E_k^{-1} T) = F(\{ E_k^{-1} g_1, E_k^{-1} g_2 \})$.   By a well-known result (see, e.g., \cite[Theorem 2.1.1, p.~31]{RA}), we have $F(\{a,b\}) = ab - a - b$, and the desired claim follows.
\end{proof}

\begin{remark}
We compute explicitly that $F_2 = 17$, $F_3 = 723$, $F_4 = 52753$, $F_5 = 49790415$, and $F_6 = 126629$.
\end{remark}


\section{The complete proof}

We are now ready to fill in the details of the proof of our 
main result, Theorem~\ref{main}.  We recall the definitions of the following
quantities that will figure in the proof:
\begin{itemize}
\item $c_k (n) = 1 + 2^n + \cdots + 2^{(k-1)n}$;
\item $E_k = \gcd(k, 2^k - 1)$ is the greatest common divisor of the set $\mS_k$ of binary $k$'th powers;
\item $F_k$ is the Frobenius number of the set $E_k^{-1} \mS_k$;
\item $d_k$ is the determinant of the Vandermonde matrix
$M_k = V(1,2,\ldots, 2^{k-1})$;
\item $\ell_k$ is the largest of the absolute values of the entries of $M_k^{-1}$
\end{itemize}

We will show that, for
$k \geq 1$, there exists a constant $W(k)$ such that 
every integer $N > F_k E_k$ that is a 
multiple of $E_k = \gcd(k, 2^k - 1)$ can be written as the sum of $W(k)$ binary $k$'th powers.

\begin{proof}
The result is clear for $k = 1$, so let us assume $k \geq 2$ and
that $N$ is a multiple of $E_k$.  
Define $Z = (F_k + 1) E_k$.  In the proof there are several
places where we need $N$ to be ``sufficiently large''; that is,
greater than some constant $C > Z$ depending only on $k$; some are awkward to write explicitly, so we do not attempt to do so.
Instead we just assume $N$ satisfies the requirement $N > C$.
The cases $F_k E_k < N \leq C$ are then handled by writing
$N$ as a sum of a constant number of elements of $\mS_k$.

Let $X := N - Z$.    Let $c$ be a constant specified below, and
let $n$ be the largest integer such that $2^n < c X^{1/k}$; we
assume $N$ is sufficiently large so that $n \geq 1$.

First we explain how to write 
$X = Y + D$, where 
\begin{enumerate}[(a)]
\item $Y < c_k (n)$; and
\item $D$ is an $\Enn$-linear combination of
$c_k (n), \ldots, c_k(n+k-1)$ with all coefficients
sufficiently large.
\end{enumerate}

To do so, define
$Q = c_k (n) + \cdots + c_k (n+k-1)$,
and $R = \lfloor X/Q \rfloor$.    We have now obtained $RQ$ (a good approximation of $X$), which is an $\Enn$-linear
combination of $c_k (n), \ldots, c_k(n+k-1)$ with every
coefficient equal to $R$.  Note that
$0 \leq X-RQ < Q$.  

We now improve this approximation of $X$ using
a greedy algorithm, as follows:  from $x-RQ$ we remove as many
copies as possible of $c_k(n + k-1)$, then as many copies as possible
of $c_k (n+k-2)$, and so forth, down to $c_k (n)$.  More precisely,
for each index
$i = k-1, k-2, \ldots, 0$ (in that order) set
$$r_i = \left\lfloor \frac{X - RQ - \sum_{i < j < k} r_j c_k(n+j)}{c_k (n+i)} \right\rfloor ,$$
and then put
$$D := RQ + r_0 c_k (n) + r_1 c_k (n+1) + \cdots + r_{k-1} c_k(n+k-1) .$$
By the way we chose the $r_i$, we have $0 \leq r_{k-1} < 2$ and
$0 \leq r_i <
c_k(n+i+1)/c_k(n+i) < 2^{k-1}$ for $0 \leq i \leq k-2$.
Furthermore, $0 \leq y < c_k (n)$.
Define $e_i = R + r_i$ for $0 \leq i < k$.
Then $D = \sum_{0 \leq i < k} e_i c_k (n+i)$.  

Since $Y < c_k (n)$, we can express $Y$ in base $2^n$ as $Y = a_0 + a_1 2^n + \cdots + a_{k-1} 2^{(k-1)n}$, where each $a_i$ is an integer satisfying $0 \leq a_i < 2^n$.

Apply the transformation discussed above in Section~\ref{vander},
obtaining the $\Que$-linear combination
$$ Y = \sum_{0 \leq i < k} b_i c_k (n+i).$$
It follows that $X = \sum_{0 \leq i < k} (e_i+b_i) c_k (n+i)$.

Furthermore, from Section~\ref{vander} we know that
each $b_i$ is at most $k\ell_k \cdot 2^n$ in absolute value,
and the denominator of each $b_i$ is at most $d_k$.

Now we want to ensure that, for $0 \leq i < k$, it holds
\begin{equation}\label{ebrightsize}
2^{n + i - 1} \leq e_i + b_i < c^\prime 2^n ,
\end{equation}
where $c^\prime > 0$ is a constant depending only on $k$.
We choose the constant $c$ mentioned above to get the
bound (\ref{ebrightsize}).

Pick $c > 0$ such that $c^{-k} = \left(2^{k-2} + k \ell_k + 1\right) 2^{k^2 - k + 1}$. Then we have
\begin{align*}
e_i + b_i &\geq R - k\ell_k \cdot 2^n > \frac{X}{Q} - k\ell_k \cdot 2^n - 1 > 
\frac{2^{kn} c^{-k}}{2 \cdot 2^{(k-1)(n+k) + 1}} - k\ell_k \cdot 2^n - 1 \\
&= 2^n \left(c^{-k} 2^{-k^2 + k - 1} - k\ell_k - 2^{-n}\right) > 2^{n + k - 2} \geq 2^{n + i - 1} ,
\end{align*}
as desired.

For the upper bound, recalling that our choice of $n$ implies that $c X^{1/k} \leq 2^{n+1}$, we have
$$ e_i + b_i < R + 2^{k-1} + k\ell_k 2^n \leq \frac{X}{Q} + 2^{k-1} + k\ell_k 2^n < \frac{2^{k(n+1)}c^{-k}}{2^{(k-1)(n + k - 1)}} + 2^{k-1} + k\ell_k 2^n,$$
and a routine calculation shows that $e_i + b_i \leq c' 2^n$,
where the $c^\prime$ depends only on $k$.  

The only problem left to resolve is that the
$e_i + b_i$ need not be integers.    Write
$X = X_1 + X_2$, where
$$X_1 := \sum_{0 \leq i < k} \lfloor e_i + b_i \rfloor c_k (n+i) .$$
Thanks to (\ref{ebrightsize}), we can use
Lemma~\ref{split} to rewrite $X_1$ as a sum of a constant
number of binary $k$'th powers.
Then
$$ X_2 = \sum_{0 \leq i < k} a_i c_k (n+i),$$
where $0 < a_i < 1$ is a rational number with denominator
$d_k$.  Writing $a_i = v_i / d_k$, we see
\begin{align*}
X_2 &= \sum_{0 \leq i < k} \frac{v_i}{d_k} 
\sum_{0 \leq j < k} 2^{(n+i)j}  \\
&= \sum_{0 \leq i, j < k} v_i \cdot \frac{2^{(n+i)j}}{d_k} \\
&= \sum_{0 \leq i, j < k} v_i  \left\lfloor 
\frac{2^{(n+i)j}}{d_k} \right\rfloor + X_3,
\end{align*}
where $0 \leq X_3 < d_k \cdot k^2$.

By Lemma~\ref{pows} we know that, provided 
$(n+i)j > kd_k + 2 \log_2 d_k$, each term
$\left\lfloor 
\frac{2^{(n+i)j}}{d_k} \right\rfloor$ is the sum
of a constant number of binary $k$'th powers, plus
an error term that is at most $2^{k d_k}$.   Thus,
provided $n$ (and hence $N$) are large enough,
this will be true for all exponents except those
corresponding to $j = 0$. Those exponents are not a problem, since for $j = 0$ we have $\left\lfloor 
\frac{2^{(n+i)j}}{d_k} \right\rfloor = 0$, because $d_k > 1$ for $k > 1$. It
follows that $X_4 := X_2-X_3$ is the sum of a constant number
of binary $k$'th powers.

Putting this all together, we have expressed 
$$ N = X_1 + X_4 + X_3 + Z,$$
where
$X_1$ and $X_4$ are both the sum of a constant
number of binary $k$'th powers, and 
$X_3 + Z$ is bounded below by $(F_k+1)E_k$ and above
by a constant.  
Now $N$, $X_1$, and $X_4$ are all multiples of
$E_k$, so the ``error term'' $X_3 + Z$ must
also be a multiple of $E_k$.  Furthermore,
the error term is larger than $E_k F_k$ and hence is
representable as a sum of a constant number of
binary $k$'th powers.
\end{proof}

\section{Final remarks}
\label{final}
Everything we have done in this paper is equally applicable to expansions in bases $b > 2$, with one
minor complication:  it may
be that if $b$ is not a prime, then the base-$b$
expansion of $1/d$ might not be purely periodic (after removing any leading zeros), but only ultimately periodic.  This adds a small complication in
Section~\ref{fractions}.  However, this case can easily be handled, and we
leave the details to the reader.

The bounds we obtained in this paper for $W(k)$ are
very weak --- at least doubly exponential --- and can certainly be improved.  We leave this as work for the future.   For example, we have

\begin{conjecture} Every natural number $> 147615$ is the sum of at most nine binary cubes.
The total number of exceptions is 4921.
\end{conjecture}

\begin{remark}
We have verified this claim up to $2^{27}$.
\end{remark}

There is another approach to Waring's theorem for binary powers that could potentially give much better bounds for $W(k)$.  
For sets $S, T \subseteq \Enn$ define the {\it sumset} $S + T$ as follows:
$$ S + T = \lbrace s + t \ : \ s \in S, t \in T \rbrace.$$  We make the following conjecture:
\begin{conjecture}
Writing $C_n$ for the 
set $\{ a\cdot c_k^2 (n) \ : \ 2^{n-1} \leq a < 2^n \}$ of cardinality $2^{n-1}$ (i.e.,
the $kn$-bit binary $k$'th powers), for $n, k \geq 1$,
all the elements in the sumset
$$C_n + C_{n+1} + \cdots + C_{n+k-1},$$ 
are actually represented {\it uniquely\/} as a sum of $k$ elements, one
chosen from each of the summands.
\end{conjecture}
If this conjecture were true --- we have proved it for
$1 \leq k \leq 3$ --- it would
prove that $\mS_k$ has positive density, and hence, by a
result of Nathanson \cite[Theorem 11.7, p.~366]{N2}, that it is an asymptotic additive basis.
From this we could obtain better bounds on $W(k)$.
\medskip

In the light of our results, it seems natural to ask about the set $\mT_1^b$ of positive integers $k$ such that $\gcd(\mS_k^b) = 1$.
Indeed, we have that the elements of $\mT_1^b$ are exactly the integers $k$ such that $\mS_k^b$ forms an asymptotic additive basis for $\mathbb{N}$.
It turn out that $\mT_1^b$ has a natural density, and even more can be said: since $\left(\frac{b^k - 1}{b - 1}\right)_{k \geq 1}$ is a Lucas sequence, we can employ the same methods of \cite{ST} to prove the following result:

\begin{theorem}
For all integers $g \geq 1$, $b \geq 2$, the set $\mT_g^b$ of positive integers $k$ such that $\gcd(\mS_k^b) = g$ has a natural density, given by
\begin{equation*}
\mathbf{d}(\mT_g^b) = \sum_{d \geq 1 \text{ coprime with } b} \frac{\mu(d)}{L_b(dg)},
\end{equation*}
where $\mu$ is the M\"obius function and $L_b(x) := \lcm(x, \ord_x(b))$, where $\ord_x(b)$ is the multiplicative order of $b$, modulo $x$.
In particular, the series converges absolutely.

Furthermore, $\mathbf{d}(\mT_g^b) > 0$ if and only if $\mT_g^b \neq \varnothing$ if and only if $g = \gcd\!\left(L_b(g), \frac{b^{L_b(g)} - 1}{b - 1}\right)$.
\end{theorem}

Also, employing the methods of \cite{LS}, the counting function of the set $\{g \geq 1 : \mT_g^b \neq \varnothing\}$ can be shown to be $\gg x / \log x$ and at most $o(x)$, as $x \to +\infty$.
Note only that, in doing so, where in \cite{LS} results of Cubre and Rouse~\cite{CubreRouse} on the density of the set of primes $p$ such that the rank of appearance of $p$ in the Fibonacci sequence is divisible by a fixed positive integer $m$ are used, one should instead use results on the density of the set of primes $p$ such that $\ord_p(b)$ is divisible by $m$ --- for example, those given by Wiertelak~\cite{Wie}.

\section*{Acknowledgment}

We are grateful to Igor Pak for introducing the first and third authors to each other.
 
\bibliographystyle{amsplain}

\end{document}